\documentclass[12pt,reqno,a4paper]{amsart}
\usepackage{amsfonts,amsmath,times}
\usepackage{a4wide}
\usepackage{tikz}

\definecolor{blau}{rgb}{0,0,0.75} %colour for in-document links
\usepackage[pdftex]{hyperref}
\hypersetup{colorlinks,linkcolor=blau,citecolor=blue,urlcolor=blau}

\allowdisplaybreaks

\newtheorem{theorem}{Theorem}
\newtheorem{lemma}{Lemma}

\newtheorem{prop}{Proposition}

\theoremstyle{definition}
\newtheorem{remark}{Remark}

\newcommand{\JAP}{\emph{Journal of Applied Probability}}

\newcommand{\PTRF}{\emph{Probability Theory and Related Fields}}

\newcommand{\mom}{\text{model}\,\ensuremath{\mathcal{M}}}
\newcommand{\mor}{\text{model}\,\ensuremath{\mathcal{R}}}

\newcommand{\momp}{\text{model}\,\ensuremath{\mathcal{M}}\,}
\newcommand{\morp}{\text{model}\,\ensuremath{\mathcal{R}}\,}

\newcommand{\ith}[1]{\ensuremath{ {#1}}th}

\newcommand{\fallfak}[2]{\ensuremath{#1^{\underline{#2}}}}

\newcommand{\Stir}[2]{\genfrac{ \{ }{ \} }{0pt}{}{#1}{#2}}
\newcommand{\stir}[2]{\genfrac{ [ }{ ] }{0pt}{}{#1}{#2}}

\newcommand{\calW}{\ensuremath{\mathcal{W}}}
\newcommand{\W}{\ensuremath{\tilde{W}}}
\newcommand{\N}{\ensuremath{\mathbb{N}}}

\newcommand{\Gro}{\ensuremath{\mathcal{O}}}
\def\P{{\mathbb {P}}}
\def\E{{\mathbb {E}}}
\def\V{{\mathbb {V}}}

\newcommand{\law}{\ensuremath{\stackrel{\mathcal{L}}=}}
\newcommand{\claw}{\ensuremath{\xrightarrow{\mathcal{L}}}}
\newcommand{\as}{\ensuremath{\xrightarrow{(a.s.)}}}

\newcommand\backward{\nabla}
\newcommand\field{\mathbb{F}}

\newcommand\given{\, \vert \, }

\newcommand\matM{{\bf M}}

\newcommand\Polya{P\' olya}

\author[M.~Kuba]{Markus Kuba}
\address{Markus Kuba\\
Institute of Applied Mathematics and Natural Sciences\\
University of Applied Sciences-Technikum Wien\\
H\"ochst\"adtplatz 5, 1200 Wien} 
\email{kuba@technikum-wien.at}

\author[H.~M.~Mahmoud]{Hosam M.~Mahmoud}
\address{Hosam M.~Mahmoud\\
Department of Statistics\\
The George Washington University, Washington, D.C. 20052, U.S.A.}
\email{hosam@gwu.edu}

\title[Affine urn models II]{Two-color balanced affine urn models with multiple drawings~II: large-index and triangular urns}
        %combinatorial probability
\begin{document}
\begin{abstract}
This is the second part of a two-part investigation. We continue the study of a class of balanced urn schemes on balls of two colors (white and black).
At each drawing, a sample of size $m\ge 1$ is drawn from the urn and ball addition rules are applied;
% rearranged the sentence for better readability 
the special case $m=1$ of sampling only a single ball coincides with ordinary balanced urn models. We consider these multiple drawings under sampling with or without replacement. For the class of  affine conditional expected value, we study the number of white balls after $n$ steps.  The affine class is parametrized by $\Lambda$, specified by the ratio of the two eigenvalues of a reduced ball replacement matrix and the sample size, leading to three different cases: small-index urns ($\Lambda \le \frac 1 2$, and the case $\Lambda= \frac 1 2$ is critical), large-index urns ($\Lambda > \frac 1 2$), and triangular urns. In Part~I we derived a central limit theorem for small index urns, and proved almost-sure convergence for large index and triangular urn models.
In the present paper (Part II), we continue the study of affiance urn schemes and study the moments of large-index urns and triangular urn models. We show moment convergence under suitable scaling and we also provide expressions for the moments.
\end{abstract}

\date{\today}
\maketitle
\section{Introduction}
Urn schemes are simple, yet versatile mathematical tools for modeling 
evolutionary processes. This is the second part of a two-part investigation of affine urn schemes.  
We discussed motivation and application in the first part; we refer the interested
reader to the introduction of that companion paper.
This investigation is devoted to the study of a generalization of the two-color
\Polya\ urn model, 
where \emph{multiple} balls are drawn at each discrete time step, 
their colors are inspected, then the entire sample is placed back in 
the urn, and rules of replacement are affected.
Such urn schemes with multiple drawing 
recently received attention~\cite{ChenWei,ChenKu2013+,JohnsonKotzMahmoud2004,KuMaPan2013+,Mah2012,Moler,Renlund,TsukijiMahmoud2001}. 
The addition/removal of balls depends on the combinations of colors in the multiset drawn. 
We suppose the two colors are white (W)
and black (B), and use the notation  $\{W^kB^{m-k}\}$
to refer to a sample of size $m$ containing $k$ white balls and $m-k$ black balls.
Specifically, we draw $m\ge 1$ balls and add/remove white and black balls 
according to the multiset $\{W^kB^{m-k}\}$ of observed colors:
If we draw $k$ white 
and $m-k$ black balls, we add $a_{m-k}$ white and $b_{m-k}$ black balls, $0\le k\le m$.
The ball replacement matrix of this urn model with multiple drawings is a rectangular $(m+1)\times 2$ matrix:
\begin{equation}
\label{MuliDrawsLinMatrix}
    \matM =
    \begin{pmatrix}
    a_0& b_0   \\
    a_1  & b_1\\
    \vdots&\vdots\\
    a_{m-1}   & b_{m-1}\\
    a_m&  b_m  \\
    \end{pmatrix}.
\end{equation}
Let $W_n$ and $B_n$ be respectively the number
of white and black urns in the urn after $n$ draws.
We assume throughout that the urn model is \emph{balanced}, 
a scheme in which one adds a constant number of balls 
$\sigma\ge 1$ each time, regardless of what multiset is withdrawn.
That is, $a_k+b_k=\sigma\ge 1$, $0\le k\le m$.
Consequently, the total number of balls $T_n=W_n+B_n$ after $n$ draws is given by a nonrandom number 
$T_n=T_0 + \sigma n$. We confine our attention to the
so-called tenable urn models, where the process of drawing and replacing balls can be continued ad infinitum.

\smallskip 

Urn models with multiple drawings and sample size $m\ge 2$ are in general more difficult to analyze compared to the ordinary case $m=1$, see~\cite{ChenWei,ChenKu2013+,Mah2012,Moler,Renlund} and the discussions therein. In Part I we carried out a structural analysis classifying urn models with multiple drawings according to the shape of the conditional expected value of the number of white balls $W_n$. In this work, as in Part I, we analyze urn schemes with multiple drawings for which the conditional expectation of the number of white balls $W_n$ after $n$ draws has an affine structure of the form
\begin{equation}
\label{MuliDrawsLinPropLinear}
\E\bigl[W_n \given \field_{n-1}\bigr]= \alpha_n W_{n-1} +\beta_n,\qquad n\ge 1.
\end{equation}
Here, $\alpha_n,\beta_n$ denote deterministic sequences certain sequences depending only on $n$, $a_{m-1}$, $a_m$, and the balance factor $\sigma$, and $\field_n$ 
denotes the sigma-algebra generated by the first $n$ draws from the urn. Note that in the case $m=1$ balanced urn models are by definition affine.
\subsection{Affine schemes}
In Part I we obtained a characterization of affine schemes. 
Basically, a scheme is {\it affine}, if and only if 
the entries in the first column satisfy the recurrence:
\begin{equation}
a_k=(m-k)(a_{m-1}-a_m)+a_m,\qquad \mbox {for \ } 0\le k\le m.
\label{Eq:ak}
\end{equation}
Equivalently, the condition can be written as 
$$a_k = hk + a_0, $$ 
for arbitrary $h$ respecting tenability.
The values $\alpha_n$ and $\beta_n$ in~\eqref{MuliDrawsLinPropLinear} are given by
\begin{equation}
\alpha_n=\frac{T_{n-1}+m(a_{m-1}-a_m)}{T_{n-1}},\qquad \beta_n=a_m,\quad n\ge 1.
\label{AffineAB}
\end{equation}
In view of the balance, the entries in the second column satisfy a similar recurrence. 

\subsection{Classification of limit laws}\label{Classification}
We also discussed in Part I a classification in terms of
the two eigenvalues $\Lambda_1$ and $\Lambda_2$ of the submatrix
$\begin{pmatrix}
a_{m-1} & b_{m-1}\\
a_m&b_m\\
\end{pmatrix}$ with $\Lambda_1$ being the larger of the two eigenvalues.\footnote{An equivalent formulation can be obtained from the first two rows of the replacement matrix; see part I and for a discussion of the higher dimensional condition.}  
The urn index is $$\Lambda = \Lambda_2/\Lambda_1=\frac{m}{\sigma}(a_{m-1}-a_m).$$ A trichotomy of cases similar to 
the classical case of sample size $m=1$ arises: 
(1) Urn schemes with $a_m\not =0$ and a small index. These have $\Lambda \le \frac12$, and the case $\Lambda=\frac 1 2$ is critical,
(2) Urn schemes with $a_m\not =0$ and a large index, $\Lambda > \frac12$,
(3) Triangular urn models with $a_m=0$.
The limiting distribution results for sample size $m=1$ and balanced urn models are well known and are, amongst others, based on the results of Bagci and Pal~\cite{Bagchi1985}, Janson~\cite{Jan2004,Jan2006}, Flajolet et al.~\cite{FlaDumPuy2006}, Chauvin et al.\cite{Chauvin1} and Neininger and Knape~\cite{NeiningerKnape}.
For a balanced two-color urn model $M	= \left(\begin{matrix} a_0 & b_0 \\ a_1 & b_1\end{matrix}\right)$ let $\Lambda$ denote the ratio of the two eigenvalues of $M$. 
For small-index urns, $\Lambda\le\frac12$, one obtains a central limit theorem for the number of white balls $\frac{W_n-\E(W_n)}{\sqrt{\V(W_n)}}\claw \mathcal{N}(0,1)$. For large-index urns, $\Lambda>\frac12$, and also triangular urns, $b_0\cdot a_1=0$, the suitably normalized (and centered) number of white balls converges almost surely to a non-normal limiting distribution.
Concerning the distribution of the limit law of large index urns there has been a \emph{flurry of activity} in the last decade. Several articles~\cite{Chauvin1,Chauvin2,Pou2008} have been entirely devoted to the study of large-index urns and the properties of the limit law such as characteristic function, moments, decomposition of random variables, fixed-point equations (smoothing transforms) etc.; see also~\cite{Jan2004,NeiningerKnape} for general studies including large-index urn models; moreover, triangular urn models and properties of its limit law, including moments and density functions, have been analyzed in~\cite{FlaDumPuy2006,Jan2006}.

\smallskip

We will extend the classification result above for $m=1$ to arbitrary affine urn models. In Part I we already discussed the case $\Lambda\le \frac12$ and proved a central limit theorem. Moreover, for $\Lambda>\frac12$ and for triangular urns almost-sure convergence of suitably defined random variables have been proved in part I using discrete martingales.
We continue our investigation analyzing the positive integer moments of the limit 
laws of large-index urns and triangular urns.
It turns out that the higher moments of the limiting distributions of large-index urn models and triangular urn models are given by nested infinite sums. 
Our results are valid for any $m\ge 1$, but in contrast to the case $m=1$ it appears that in general for $m\ge 2$ the nested sums cannot be simplified.

Note that concerning urn models with multiple drawings and replacement matrix $M$ as given by~\eqref{MuliDrawsLinMatrix} we call an urn model \emph{triangular} if $a_m=0$ or $b_0=0$ or both $a_m=b_0=0$. The case $b_0=0$ and $a_m>0$ for the black balls corresponds to the case $a_m=0$ and $b_0>0$ for the white balls; by the relation $B_n=T_n-W_n$ this implies that, without loss of generality, we can restrict our attention to triangular urns with $a_m=0$ and $b_0\ge 0$. If both $a_m=b_0=0$ one obtains the so-called P\'olya urn model as treated in~\cite{ChenKu2013+}. Setting $a_{m-1}=c>0$ triangular urns in the affine scheme with $a_m=0$ are specified by a rectangular $(m+1)\times 2$ matrix:
\begin{equation*}
   \matM =
    \begin{pmatrix}
    mc& \sigma-mc   \\
    (m-1)c  & \sigma-(m-1)c\\
    \hdots&\hdots\\
    c   & \sigma-c\\
    0&  \sigma  \\
    \end{pmatrix},
\end{equation*}
and the three parameters $a_{m-1}=c>0$, the sample size $m$ and the total balance $\sigma>0$ such that $\sigma\ge mc$. The special case $\sigma=mc$ corresponds to generalized P\'olya urn model, as discussed in~\cite{ChenKu2013+}. We will improve the results of~\cite{ChenKu2013+} by obtaining an explicit nonrecursive formula for the moments of $W_n$, as well as for the limiting distribution.

\subsection{Affine urn models}\label{AffineUrns}
A direct consequence of the affine expectation~\eqref{MuliDrawsLinPropLinear} is a martingale structure, and a closed form expression for the expected value and the second moment.
We collect the results relevant for this work in the following Lemma.
\begin{lemma}[\cite{KuMa201314}]
\label{Prop:caligraphicW}
The expected value of the random variable $W_{n}$, counting the number 
of white balls in a two-color balanced affine urn model with multiple drawings, is for both sampling models 
$\mathcal{M}$ and $\mathcal{R}$ given by $\E[W_{n}]=\frac{a_m}{g_n}\sum_{j=1}^{n}g_j +W_0\frac{1}{g_n}$. Here, we have
\begin{equation}
\label{ExpansionGn}
g_n=\prod_{j=0}^{n-1}\frac{T_j}{T_j+m(a_{m-1}-a_m)}=\frac{\binom{n-1+\frac{T_0}\sigma}{n}}{\binom{n-1+\frac{T_0}\sigma +\Lambda}{n}}=\frac{\Gamma(\frac{T_0}\sigma+\Lambda)}{\Gamma(\frac{T_0}\sigma)}n^{-\Lambda}\Big(1+O\Bigl(\frac1{n}\Bigr)\Bigr).
\end{equation}
For $a_m\neq 0$ and $\Lambda<1$, including large-index urns, we have 
\begin{align*}
\E[W_{n}]&= \frac{a_m(n+\frac{T_0}\sigma)}{1-\Lambda} + \Big(W_0-\frac{\frac{a_mT_0}\sigma}{1-\Lambda}\Big)\frac{\binom{n-1+\frac{T_0}\sigma + \Lambda}{n}}{\binom{n-1
+\frac{T_0}\sigma}{n}} \\
&= \frac{a_m }{1-\Lambda}\, n + \Big( W_0-\frac{\frac{a_mT_0}\sigma}{1-\Lambda}\Big) \frac{\Gamma(\frac{T_0}\sigma)}{\Gamma(\frac{T_0}\sigma+\Lambda)} \, n^{\Lambda} 
+ \Gro(1).
\end{align*}
For triangular urns with $a_m=0$ we have the closed form expression 
$$\E[W_{n}]=W_0\, \frac{n\sigma +T_0}{T_0}.$$

The random variable $\calW_n= g_n ( W_n- \E[W_n])$ is a centered martingale with respect to the natural filtration: $\E[\calW_n\given \field_{n-1}]=\calW_{n-1}$, $n\ge 1$, with $\mathcal{W}_{0}=0$.
For large-index urns, $\calW_n$ convergences almost surely and in $L_2$ to a limit $\calW_\infty$. For triangular urn models, where $a_m=0$, the random variable $\mathfrak{W}_n=g_n W_n$ is a nonnegative martingale and converges almost surely to a limit $\mathfrak{W}_\infty$.
\end{lemma}
\begin{remark}
\label{Complication}
A slight unpleasant complication is the case distinction for the expected value between $a_m\neq 0$ and $a_m=0$. The case $a_m=0$ leads to $\Lambda=\frac{m a_{m-1}}{\sigma}$, which may be equal to one for P\'olya urn models with $a_{m-1}=c$ and $\sigma=m c$. However, we can interpret the explicit formula $\E[W_n]$ stated for $a_{m}\neq 0$ and $\Lambda<1$ in the right way: 
We set {\it first} $a_m=0$, so all terms which include $\frac1{1-\Lambda}$ vanish, and only afterward we set $\Lambda$ to its corresponding value, including the case $\Lambda=1$. In other words, the quotient $\frac{a_m}{1-\Lambda}$ is zero for $a_m=0$, regardless of the value of $\Lambda \le 1$.
\end{remark}

\subsection{Plan of the paper and notation}
In Part I we obtained Gaussian limits for small-index urn schemes and almost-sure limits for triangular and large-index urns. 
It is our aim to complete the study of triangular and large-index urns. We study the (positive integer) moments $W_n^s$, $s\in\N$, of~$W_n$ and the moments of the limit laws $\calW_\infty$ for large-index urns and $\mathfrak{W}_\infty$ for triangular urns. We provide a recursive characterization for the moments of the limiting distribution. 
For triangular urns $a_m=0$ and $b_0\ge 0$, we generalize the existing results concerning the case $a_m=b_0=0$~(see~\cite{ChenKu2013+}) obtaining an explicit nonrecursive descriptions of the moments.

We denote by $\fallfak{x}{k}$ the $k$th falling factorial, $x(x-1)\dots (x-k+1)$, $k\ge 0$, with $\fallfak{x}0=1$.
We shall also use $\backward$, the backward difference operator, defined by $\backward h_n = h_n -h_{n-1}$, when acting on a function $h_n$.
We use $\Stir{s}{k}$ to denote the Stirling numbers of the second kind, and $\stir{s}{k}$ to denote the unsigned Stirling numbers of the first kind (see~\cite{Stanley} or~\cite{GraKnuPa});
these numbers appear as coefficients in the expansions
\begin{equation*}
x^s=\sum_{k=0}^{s}\Stir{s}k \fallfak{x}k,\qquad \fallfak{x}s=\sum_{k=0}^{s}(-1)^{s-k}\stir{s}k x^k,\qquad 
\end{equation*}
relating ordinary powers $x^s$ to the falling factorials $\fallfak{x}{s}$. 
Moreover, in this article we refer with $\mathcal{W}_\infty$ to the almost-sure limit of~$W_n$ for large-index urns with $\frac12<\Lambda<1$ and with $\mathfrak{W}_\infty$ 
to the almost-sure limit for 
triangular urns satisfying $a_m=0$ and $b_0\ge 0$.

\section{Preliminaries}
\label{Sec:Prelim}
\subsection{Sampling schemes}
\label{Sec:PrelimSampling}
Assume that an urn contains $w$ white and $b$ black balls. We consider two different sampling schemes for drawing the $m$ balls at each step: \mom\ and \mor. In \mom, we draw the $m$ balls without replacement. The $m$ balls are drawn at once and their colors are examined. 
After the sample is collected, we put the entire sample back in the urn and execute the replacement rules according to the counts of colors observed. 
The tenability assumption implies that for \mom\ the coefficients $a_k$ of the 
ball replacement matrix~\eqref{MuliDrawsLinMatrix} 
satisfy the condition 
$a_k\ge -(m-k)$,\footnote{These assumptions can be relaxed a little bit,
if the initial values $W_0$ and $B_0$ are adapted to the entries in the ball replacement matrix. E.g., for $m=1$ the urn model with ball replacement matrix
$\left(\begin{smallmatrix}-3 & 8\\
6& -4
\end{smallmatrix}\right)$ is still tenable if $W_0$ is a multiple of $3$ and $B_0$ a multiple of $4$.} for $0\le k\le m$. 
Without loss of generality we assume throughout this work that the initial number of balls $T_0=W_0+B_0\ge m$.

The probability $\P(W^kB^{m-k})$ of drawing $k$ white and $m-k$ black balls is given by 
$$\P(W^kB^{m-k})= \frac1{\fallfak{(b+w)}{m}}\binom{m}{k}\fallfak{w}{k}\, \fallfak{b}{m-k}=\frac{\binom{w}k\binom{b}{m-k}}{\binom{b+w}m},\qquad 0\le k\le m.$$
Thus $X$, the number of white balls in the sample,
follows a hypergeometric distribution,
with parameters $w+b, w$, and $m$, that is, 
one that counts the number of white balls in a sample of size~$m$ balls 
taken out of an urn containing $w$ white and $b$ black balls (a total of
$\tau = w+b$ balls).
The first two moments of $X$ are given by
$$\E[X]=m\frac w \tau,\qquad \E[X^2] = \frac{w(w-1)m(m-1)}{\tau(\tau-1)}
         + \frac{w m}\tau.$$
The $\ell$th moments of $X$ can be written as a polynomial in the parameter $w$:
$$ \E[X^{\ell}]=\sum_{k=0}^{m} k^\ell \frac{\binom{w}k\binom{\tau-w}{m-k}}{\binom{\tau}{m}}
=\sum_{i=0}^{\ell^{*}}w^i \sum_{j=i}^{\ell^{*}}(-1)^{j-i} \frac{\stir{j}i \Stir{\ell}{j}\binom{m}{j}}{\binom{\tau}{j}},$$
with $\ell^{*}=\min\{\ell,m\}$, for arbitrary $\ell\ge 0$. 

In \mor, we draw the $m$ balls with replacement. The $m$ balls are drawn one at a time. After a ball is drawn,  
its color is observed, and is reinserted in the urn, and thus 
it might reappear in the sampling of
one multiset. After $m$ balls are collected in this way (and they are all back in the urn),
we execute the replacement rules according to the counts of colors observed.
By the tenability assumption $a_k\ge -1$ for $0\le k\le m-1$ and $a_m\ge 0$ for \mor.

The probability $\P(W^kB^{m-k})$ of drawing $k$ 
white and $m-k$ black balls is given by 
$$\P(W^kB^{m-k}) = \frac1{(b+w)^m}\binom{m}{k}w^k \, b^{m-k},\qquad 0\le k\le m.$$
In other words, under \morp, the number of white balls in the multiset 
of size $m$ follows a binomial distribution with parameters $m$, and ${w}/{\tau}$, one that counts the number of successes in~$m$ 
independent identically distributed experiments, with  $w/\tau$ probability 
of success per experiment.
Let $Y$ denote such a binomially distributed 
random variable. 
Then, the first two moments of $Y$ are given by
$$\E[Y]=m\frac w \tau,\qquad \E[Y^2]= m\frac w \tau\Big(1-\frac w \tau\Big) + m^2\frac{w^2}{\tau^2}.$$
The $\ell$th moments of $Y$ can be written as a polynomial in the parameter $w$:
$$\E[Y^\ell]
=\sum_{k=0}^{m} k^\ell \binom{m}k
\frac{w^k(\tau-w)^{m-k} }{\tau^m}
=\sum_{j=0}^{\ell}\Stir{\ell}j \fallfak{m}{j}\frac{w^j}{\tau^j}.$$
\subsection{Distributional equations}
In what follows, 
we use the notation 
$\mathbb{I}_{n}(W^kB^{m-k})$ to stand for the indicator
of the event that the multiset $\{W^kB^{m-k}\}$ is drawn in the
$n$th sampling.
Conditioning on the outcome of the $n$th draw, 
we obtain a distributional equation for $W_n$.
The number of white balls after $n$ draws is the number of white balls after $n-1$ draws, plus the contribution of white balls
after the $n$th sample is obtained (with \ $\law$ \ for equality in law):
\begin{equation}
\label{MuliDrawsLinDistEqn1}
W_{n}\ \law\ W_{n-1}  + \sum_{k=0}^{m} a_{m-k} \, \mathbb{I}_{n}(W^kB^{m-k}),\quad n\ge 1.
\end{equation}
Let $\field_{n-1}$ denote the $\sigma$-field generated by the first $n-1$ draws. 
The indicators $\mathbb{I}_{n}(W^kB^{m-k})$ satisfy
\begin{equation}
\label{MuliDrawsLinDistEqn2}
\P\bigl(\mathbb{I}_{n}(W^kB^{m-k})=1\given \field_{n-1}\bigr)=\frac{\binom{W_{n-1}}{k}\binom{B_{n-1}}{m-k} }{\binom{T_{n-1}}{m}}=\frac{\binom{W_{n-1}}{k}\binom{T_{n-1}-W_{n-1}}{m-k} }{\binom{T_{n-1}}{m}}
\end{equation}
for \mom, and 
\begin{equation}
\label{MuliDrawsLinDistEqn3}
\P\bigl(\mathbb{I}_{n}(W^kB^{m-k})=1\given \field_{n-1}\bigr)=\binom{m}{k}\frac{W_{n-1}^k B_{n-1}^{m-k} }{T_{n-1}^m}=\binom{m}{k}\frac{W_{n-1}^k(T_{n-1}-W_{n-1})^{m-k} }{T_{n-1}^{m}}
\end{equation}
for \mor. We obtain for $W_n^{s}$, $s\ge 1$, a distributional equation by taking the $s$th power 
of~\eqref{MuliDrawsLinDistEqn1}, and using the fact that the indicator variables are mutually exclusive:
\begin{equation}
\label{MuliDrawsLinDistEqn4}
W_{n}^s\ \law  \ \sum_{\ell=0}^{s}\binom{s}\ell W_{n-1}^{s-\ell}\sum_{k=0}^{m} a_{m-k}^{\ell} \, \mathbb{I}_{n}(W^kB^{m-k})
,\qquad n\ge 1.
\end{equation}

\section{Moment structure}
In order to study the moments of $W_n$ and of the almost sure limits $\mathcal{W}_\infty$, $\mathfrak{W}_\infty${\color{red}, }
  we analyze higher shifted moments. This will enable us to provide a recursive characterization of the moments of both. We also complement these result by obtaining concrete explicit expressions for the moments.

\subsection{Higher moments: exact representations\label{LargeMom}}
The asymptotic expansion of the expected value suggests that we shift $W_n$ by the dominant term $\frac{a_m}{1-\Lambda}\, n$ of its asymptotic expansion~\eqref{Prop:caligraphicW}.
Thus, we consider the shifted random variable $\W_n$ defined by
\begin{equation}
\W_n=W_n-\frac{a_m }{1-\Lambda}\, n,
\label{ShiftDef}
\end{equation}
which is well defined for arbitrary $\Lambda<1$. For triangular urn models $a_m=0$, so
\begin{equation}
\W_n=W_n.
\label{ShiftDefTriangular}
\end{equation}

As discussed in Remark~\ref{Complication}, this random variable has to be interpreted in the right way in the P\'olya urn case when $a_m=0$ and $\Lambda=1$, interpreting the fraction 
$\frac{a_m}{1-\Lambda}$ as zero, so $\W_n=W_n$,  regardless of the value of $\Lambda$. We obtain the following recurrence relation for the moments $\E[\W_n^s]$.
\begin{lemma}
\label{HIGHMOMlemma1}
The positive integer moments $\E[\W_n^s]$ of the shifted random variable $\W_n=W_n-\frac{a_m }{1-\Lambda}\, n$, with $W_n$ counting the number 
of white balls in a two-color balanced affine urn model with multiple drawings, satisfy for both sampling models 
$\mathcal{M}$ and $\mathcal{R}$ the recurrence relation
\[
\E[\W_n^s]=\sum_{r=0}^{s}f_{n,s,r} \, \E[\W_{n-1}^r],\quad n\ge 1,
\]
with initial values $\E[\W_0^s]=W_0^s$, and the values $f_{n,s,r}$ being given by 
\begin{equation*}
\begin{split}
f_{n,s,r}&=\sum_{j=s-r}^{s}\binom{s}j \Lambda^j\sum_{\ell=j-(s-r)}^{j}\binom{\ell}{j-(s-r)}
\Big(\frac{a_m(n-1)}{1-\Lambda}\Big)^{\ell-j+s-r} \\
&\qquad \qquad {} \times\sum_{i=\ell}^{j}\binom{j}{i}\frac{\sigma^i a_m^{j-i}}{m^i(\Lambda-1)^{j-i}}\, p_{n;(i,\ell)},
\end{split}
\end{equation*}
for $0\le r\le s$ with
\[
p_{n;(i,\ell)}=
\begin{cases}
\sum_{h=\ell}^{i^{*}}(-1)^{h-i}\frac{\stir{h}\ell \Stir{i}h \binom{m}h}{\binom{T_{n-1}}{h}},& \mom;\\
\Stir{i}\ell \frac{\fallfak{m}{\ell}}{T_{n-1}^\ell},&\mor.
\end{cases}
\]
with $i^*=\min\{i,m\}$. 
\end{lemma} 
\begin{remark}
\label{RemarkSimplyTriangularF}
One readily obtains concrete expressions for $f_{n,s,r}$ and $s=1,2,\dots$ using the formula above and preferentially a Computer Algebra System. For example, for $s=1$ we obtain the model-independent result
$$
%f_{n,1,0}=a_m\Lambda\Big(\frac{(n-1)\sigma}{(1-\Lambda)T_{n-1}}+\frac{1}{\Lambda-1}\Big) 
%after simplifications we obtain
f_{n,1,0}=-\frac{a_m\Lambda T_0 }{(1-\Lambda)T_{n-1}},\qquad f_{n,1,1}=1+\frac{\sigma\Lambda}{T_{n-1}}.
$$ 
Note that for $s\ge 2$ the values $f_{n,s,r}$ are model-dependent. For triangular urns, the values $f_{n,s,r}$ can be simplified: One readily observes that $f_{n,s,0}=0$, and for $1\le r\le s$ we have
\begin{equation*}
f_{n,s,r}=\sum_{j=s-r}^s \binom{s}j \frac{\Lambda^j\sigma^j}{ m^j} \, p_{n;(j,j-(s-r))}.
\end{equation*}
\end{remark}
\begin{proof}
From~\eqref{Eq:ak} and \eqref{MuliDrawsLinDistEqn1} 
%Proposition~\ref{MuliDrawsLinPropLinear} 
%I think the new (3) is more relevant.
we obtain the distributional equation
\begin{equation}
\label{CenteredMomEqn1}
\W_n\law \W_{n-1}+\frac{\sigma\Lambda}{m}\sum_{k=0}^{m}k \mathbb{I}_{n}(W^kB^{m-k})+a_m\frac{\Lambda}{\Lambda-1}.
\end{equation}
Taking the \ith{s} power in~\eqref{CenteredMomEqn1} leads to
%to:
\begin{equation*}
\begin{split}
%\label{CenteredMomEqn3}
\W_n^s&\law \sum_{j=0}^{s}\binom{s}j\W_{n-1}^{s-j}\Lambda^j\bigg[\sum_{i=0}^{j}\binom{j}i \frac{\sigma^ia_m^{j-i}}{m^i(\Lambda-1)^{j-i}}
\sum_{k=0}^{m}k^i \mathbb{I}_{n}(W^kB^{m-k})\bigg].
\end{split}
\end{equation*}
We take the conditional expectation and simplify the sum
\[
\sum_{k=0}^{m}k^i \, \E\big[ \mathbb{I}_{n}(W^kB^{m-k})\mid \field_{n-1}\big].
\]
By~\eqref{MuliDrawsLinDistEqn2} and~\eqref{MuliDrawsLinDistEqn3} and the properties 
of the binomial and hypergeometric distributions we have
\[
\sum_{k=0}^{m}k^i\, \E\big[ \mathbb{I}_{n}(W^kB^{m-k}\mid \field_{n-1}\big]
=\sum_{\ell=0}^{i}p_{n;(i,\ell)}W_{n-1}^\ell, 
\]
with $p_{n;(i,\ell)}$ as given in Lemma~\ref{HIGHMOMlemma1}.
Finally, converting $W_{n-1}^\ell$ into powers of $\W_{n-1}$ and several summation changes lead to the stated result.
\end{proof}

A direct consequence of Lemma~\ref{HIGHMOMlemma1} is an expression of the 
moments $\E[\W_n^s]$ in terms of the moments $\E[\W_\ell^r]$, $1\le \ell\le n-1$ and $0\le r\le s-1$.
\begin{prop}
\label{CenteredMomProp}
The positive integer moments $\E[\W_n^s]$ satisfy the recurrence relation
\[
\E[\W_n^s]= \bigg(\prod_{j=1}^n f_{j,s,s}\bigg)
\bigg(W_0^s+\sum_{\ell=1}^n \frac{\sum_{i=0}^{s-1}f_{\ell,s,i}\E[\W_{\ell-1}^i]}{\prod_{j=1}^\ell f_{j,s,s}}\bigg),
\]
with $f_{n,s,r}$ as given in Lemma~\ref{HIGHMOMlemma1};
in particular 
\[
f_{n,s,s}=\sum_{j=0}^{s^*} \binom{s}j C^j\, p_{n;(j,j)},\quad\text{with}\quad C=\frac{\Lambda\sigma}{m},
\]
and $s^*=\min\{s,m\}$.
\end{prop}
\begin{proof}
Lemma~\ref{HIGHMOMlemma1} states that 
\[
\E[\W_n^s]=\sum_{r=0}^{s}f_{n,s,r} \E[\W_{n-1}^r] = f_{n,s,s}\E[\W_{n-1}^s]+\sum_{r=0}^{s-1}f_{n,s,r} \E[\W_{n-1}^r]. 
\]
Consequently, we can write
\[
\frac{\E[\W_n^s]}{\prod_{j=1}^{n}f_{j,s,s}}=\frac{\E[\W_{n-1}^s]}{\prod_{j=1}^{n-1}f_{j,s,s}}+\sum_{r=0}^{s-1}\frac{f_{n,s,r} \E[\W_{n-1}^r]}{\prod_{j=1}^{n}f_{j,s,s}}.
\]
This implies that
\[
\frac{\E[\W_n^s]}{\prod_{j=1}^{n}f_{j,s,s}}=\E[\W_{0}^s]+\sum_{\ell=1}^{n}\sum_{r=0}^{s-1}\frac{f_{\ell,s,r} \E[\W_{\ell-1}^r]}{\prod_{j=1}^{\ell}f_{j,s,s}}.
\]
Multiplication with $\prod_{j=1}^{n}f_{j,s,s}$ leads to the stated result.
\end{proof}
\subsection{Limits of the shifted moments}
Using Proposition~\ref{CenteredMomProp}, we can derive asymptotic expansions of higher moments of $\W_n$; compare 
with~\cite{KuMaPan2013+}, where a special model of small-index urns was treated similarly. This would also allow to strengthen our previous results concerning the central limit theorems for small-index urns, $\Lambda\le \frac12$, adding convergence of positive integer moments of the number of white balls to the moments of the normal distribution. We omit the involved computational details. As a first application of Proposition~\ref{CenteredMomProp}, we prove the existence of the limits of the normalized moments of $\W_n=W_n-\frac{a_m}{1-\Lambda}\, n$, for large-index urns and also for triangular urn models.

\begin{theorem}\label{TheShiftMoms}
For large urn models with $\Lambda>\frac12$ or triangular urn models with $a_m=0$ the moments $E_s=\lim_{n\to\infty}\E\Big[\frac{\W_n^s}{n^{s\Lambda}}\Big]$ exist. For $s=1$ we obtain
$$
E_1=\Big( W_0-\frac{\frac{a_m T_0}\sigma}{1-\Lambda}\Big) \frac{\Gamma(\frac{T_0}\sigma)}{\Gamma(\frac{T_0}\sigma+\Lambda)}.
$$
For the higher moments $E_s$, $s\ge 1$, we obtain for \momp the expressions 
\[
E_s= 
\bigg( \prod_{\ell=1}^{s^*}\frac{\Gamma(\frac{T_0+1-\ell}{\sigma})}{\Gamma(\lambda_{\ell,s})}\bigg)
\bigg(W_0^s+\sum_{\ell=1}^\infty \frac{\sum_{i=0}^{s-1}f_{\ell,s,i}\E[\W_{\ell-1}^i]}{\prod_{j=1}^\ell f_{j,s,s}}\bigg),
\]
with $s^*=\min\{s,m\}$, and $\lambda_{\ell,s}$ denoting the negated roots %(times minus one) 
of the monic polynomials
\[
P_s(x)=\frac{s^*!}{\sigma^{s^*}}\sum_{\ell=0}^{s^*}\frac{\binom{s}{\ell}}{\binom{s^*}\ell}C^\ell\binom{m}\ell\, \binom{x \sigma+T_0-\ell}{s^*-\ell}
%= \prod_{j=1}^{s^*}(x+\lambda_{\ell,s})
,\quad
C=\frac{\Lambda\sigma}{m}.
\]
For \morp we obtain 
\[
E_s=
\bigg( \frac{\Gamma^s(\frac{T_0}{\sigma})}{\prod_{\ell=1}^{s}\Gamma(\mu_{\ell,s})}\bigg)
\bigg(W_0^s+\sum_{\ell=1}^\infty \frac{\sum_{i=0}^{s-1}f_{\ell,s,i}\E[\W_{\ell-1}^i]}{\prod_{j=1}^\ell f_{j,s,s}}\bigg),
\]
with $\mu_{\ell,s}$ denoting the negated roots
of the polynomials
\[
Q_s(x)=\frac{1}{\sigma^s}\sum_{\ell=0}^s \binom{s}\ell C^\ell \binom{m}\ell \ell!\, (x\sigma+T_0)^{s-\ell}
%= \prod_{j=1}^{s}(x+\mu_{\ell,s})
.
\]
\end{theorem}
\begin{remark}
The result above implies that for $s\ge 2$ the limits $E_s$ of the normalized moments of $\W_n=W_n-\frac{a_m}{1-\Lambda}\, n$ can be expressed as nested infinite sums. 
In particular, the second moments are for both models readily obtained as a double sum, additionally using Proposition~\ref{CenteredMomProp} for $f_{n,s,r}$ and Lemma~\ref{Prop:caligraphicW}
for $\E[\tilde{W}_n]$. This allows to obtain a closed form expression for the variance of $\W_n/n^{\Lambda}$, and also its limit $E_2-E_1^2>0$. 
For $m\ge 2$ it seems that these nested infinite sums cannot be simplified in general due to the structure of the monic polynomials $P_s(x)$ and $Q_s(x)$ and their negated roots. In contrast, for $m=1$ the roots are much simpler and simplifications occur; compare with the discussion in~\cite{ChenKu2013+}.
\end{remark}

\begin{proof}[Proof]
In order to prove the existence of $E_k$, $k\ge 1$ we first turn to the case $k=1$. By Lemma~\ref{Prop:caligraphicW} we readily obtain $\E[\tilde{W}_n]$ and the stated result for $E_1$.
Next we use Proposition~\ref{CenteredMomProp} and we proceed in two steps. First, we derive an asymptotic expansion of the products $\prod_{j=1}^n f_{j,s,s}$ as given in Lemma~\ref{HIGHMOMlemma1}. Then, we inductively prove the existence of the infinite sums in the expressions for $E_s$, $s>1$.
Concerning the asymptotic expansions of $\prod_{j=1}^n f_{j,s,s}$, we consider for \mom\
the relation
\[
f_{j,s,s}=\sum_{\ell=0}^{s^*} \binom{s}\ell C^\ell \frac{\binom{m}\ell}{\binom{T_{j-1}}\ell}=
\frac{\sum_{\ell=0}^{s^*}\frac{\binom{s}\ell}{\binom{s^*}\ell} C^\ell \binom{m}\ell \binom{T_{j-1}-\ell}{s^*-\ell}}{\binom{T_{j-1}}{s^*}}
=\frac{\sum_{\ell=0}^{s^*}\frac{\binom{s}\ell}{\binom{s^*}\ell} C^\ell \binom{m}\ell \binom{(j-1)\sigma+T_0-\ell}{s^*-\ell}}{\binom{(j-1)\sigma+T_0}{s^*}}. 
\]
The monic polynomials $P_s(x)$, as stated in Theorem~\ref{TheShiftMoms}, are related to $ f_{j,s,s}$ 
%\[
%P_s(x)=\frac{s!}{\sigma^s}\sum_{\ell=0}^{s}c^\ell\binom{m}\ell\binom{x \sigma+T_0-\ell}{s-\ell}.
%\]
by
\[
f_{j,s,s}=\frac{P_s(j-1)}{\prod_{\ell=1}^{s^*}(j-1+\frac{T_0+1-\ell}{\sigma})}.
\]
Let $\lambda_{\ell,s}$ denote the negated roots of the equation $P_s(x)=0$, such that $P_s(x)=\prod_{\ell=1}^{s^*}(x+\lambda_{\ell,s})$. 
Consequently,
\[
f_{j,s,s}=\frac{\prod_{\ell=1}^{s^*}(j-1+\lambda_{\ell,s})}{\prod_{\ell=1}^{s^*}(j-1+\frac{T_0+1-\ell}{\sigma})},\qquad\text{and } \qquad
\prod_{j=1}^{n}f_{j,s,s}=\prod_{\ell=1}^{s^*} \frac{\Gamma(n+\lambda_{\ell,s})\, \Gamma(\frac{T_0+1-\ell}{\sigma})}{\Gamma(\lambda_{\ell,s})\, \Gamma(n+\frac{T_0+1-\ell}{\sigma})}.
\]
Note that by definition, $f_{j,s,s}>0$ for all $j\in\N$. Hence, the negated roots $\lambda_{\ell,s}\neq -n$, for all $n\in\N$ and the expression above is well defined.
By Stirling approximation, we obtain the asymptotic expansion
\begin{equation*}
\prod_{j=1}^{n}f_{j,s,s}=\frac{n^{\sum_{\ell=1}^{s^*}\lambda_{\ell,s}}}{n^{\frac{s^* T_0-\binom{s^*}{2}}{\sigma}}}\frac{\prod_{\ell=1}^{s^*}\Gamma(\frac{T_0+1-\ell}{\sigma})}{\Gamma(\lambda_{\ell,s})}\Big(1+\mathcal{O}\Bigl(\frac1n\Bigr)\Bigr).
\end{equation*}
Let $[x^k]$ denote the extraction of coefficients operator. Since
\begin{align*}
\sum_{\ell=1}^{s}\lambda_{\ell,s}&=[x^{s^*-1}]\prod_{\ell=1}^{s^*}(x+\lambda_{\ell,s})=[x^{s^*-1}]\frac{s^*!}{\sigma^{s^*}}\sum_{\ell=0}^{s^*}\frac{\binom{s}{\ell}}{\binom{s^*}{\ell}}C^\ell\binom{m}\ell\binom{x \sigma+T_0-\ell}{s^*-\ell}\\
&=\frac{s^*!}{\sigma^{s^*}}[x^{s^*-1}]\bigg(\binom{ x \sigma+T_0}{s^*} + \frac{s}{s^*}\, Cm\binom{x \sigma+T_0-1}{s^*-1}\bigg)\\
&=\frac{s^*!}{\sigma^{s^*}}\Big(\frac{\sigma^{s^*-1}(s T_0-\binom{s^*}{2})}{s^*!}+\frac{sCm\sigma^{s^*-1}}{s^*(s^*-1)!}\Big) \\
&=\frac{s^*T_0-\binom{s^*}2}{\sigma} + \frac{Cm s}{\sigma}\\
&=\frac{s^* T_0-\binom{s^*}2}{\sigma} + \Lambda s,
\end{align*}
it follows that $n^{\sum_{\ell=1}^{s}\lambda_{\ell,s}}=n^{\Lambda s} n^{\frac{s^* T_0-\binom{s^*}2}{\sigma}}$.
Hence, we get for \momp the asymptotic expansion:
\begin{equation*}
\prod_{j=1}^{n}f_{j,s,s}=n^{\Lambda s} \prod_{\ell=1}^{s^*}\frac{\Gamma(\frac{T_0+1-\ell}{\sigma})}{\Gamma(\lambda_{\ell,s})}\Big(1+\mathcal{O}\Bigl(\frac1n\Bigr)\Bigr).
\end{equation*}

Concerning \morp we can proceed in a similar fashion. We have
\[
f_{j,s,s}=\sum_{\ell=0}^s \binom{s}\ell c^\ell \frac{\fallfak{m}{\ell}}{T_{j-1}^{\ell}}=\frac{\sum_{\ell=0}^s \binom{s}\ell c^\ell \binom{m}\ell \ell! \,T_{j-1}^{s-\ell}}{T_{j-1}^{s}}
=\frac{\sum_{\ell=0}^s \binom{s}\ell c^\ell \binom{m}\ell \ell!\, ((j-1)\sigma+T_0)^{s-\ell}}{((j-1)\sigma+T_0)^{s}},
\]
and consider the monic polynomials $Q_s(x)$ with
\[
f_{j,s,s}=\frac{Q_s(j-1)}{(j-1+\frac{T_0}{\sigma})^s}.
\] 
Using arguments similar to \momp we obtain 
\begin{equation*}
\prod_{j=1}^{n}f_{j,s,s}=n^{\Lambda s}  \frac{\Gamma^s(\frac{T_0}{\sigma})}{\prod_{\ell=1}^{s}\Gamma(\mu_{\ell,s})}\Big(1+\mathcal{O}\Bigl(\frac1n\Bigr)\Bigr).
\end{equation*}
It remains to prove the existence of the moments $E_s=\lim_{n\to\infty}\E\Big[\frac{\W_n^s}{n^{s\Lambda}}\Big]$. 
By the asymptotic expansions of $\prod_{j=1}^{n}f_{j,s,s}$ the limit $\lim_{n\to\infty}\frac{\prod_{j=1}^n f_{j,s,s}}{n^{s\Lambda}}$ exists, and we have 
\[
\lim_{n\to\infty}\E\Big[\frac{\W_n^s}{n^{s\Lambda}}\Big]= \lim_{n\to\infty}\frac{\prod_{j=1}^n f_{j,s,s}}{n^{s\Lambda}}
\bigg(W_0^s+\sum_{\ell=1}^\infty \frac{\sum_{i=0}^{s-1}f_{\ell,s,i}\E[\W_{\ell-1}^i]}{\prod_{j=1}^\ell f_{j,s,s}}\bigg).
\]
We have to show that the sums 
\[
\sum_{\ell=1}^n \frac{\sum_{i=0}^{s-1}f_{\ell,s,i}\E[\W_{\ell-1}^i]}{\prod_{j=1}^\ell f_{j,s,s}}
\]
are convergent, for $n\to\infty$. Assume inductively that the moments $\E[\W_n^s]$ satisfy $\E[\W_n^s]\sim \kappa_s n^{s\Lambda}$ for some values $\kappa_s$. For $s=1$, this readily follows from the asymptotic expansion of $\E[W_n]$ in Lemma~\ref{Prop:caligraphicW}. For $s>1$ we analyze the recurrence relation stated in Proposition~\ref{CenteredMomProp}: 
\[
\E[\W_n^s]= \bigg(\prod_{j=1}^n f_{j,s,s}\bigg)
\bigg(W_0^s+\sum_{\ell=1}^n \frac{\sum_{i=0}^{s-1}f_{\ell,s,i}\E[\W_{\ell-1}^i]}{\prod_{j=1}^\ell f_{j,s,s}}\bigg).
\]
We already know that there exist constants $c_s>0$, such that $\prod_{j=1}^{n}f_{j,s,s}>c_s n^{\Lambda s}$, $s\in\N$.
Since $p_{n,(i,\ell)}=\mathcal{O}(\frac{1}{T_{n-1}^{\ell}})=\mathcal{O}(\frac{1}{n^{\ell}})${\color{red},} we obtain the crude bound $f_{n,s,r}=\mathcal{O}(1)$.
By our induction assumption $\E[\W_n^r]\sim \kappa_r n^{r\Lambda}$ for all $1\le r<s$.
We split the sum $\sum_{i=0}^{s-1}f_{\ell,s,i}\E[\W_{\ell-1}^i]$ into two parts:
$\sum_{i=0}^{s-2}f_{\ell,s,i}\E[\W_{\ell-1}^i]$ and $f_{\ell,s,s-1}\E[\W_{\ell-1}^{s-1}]$.
Hence, for large-index urn models with $\Lambda>\frac12${\color{red},} we obtain for the first part
%removed a line and the display fit on one line
\begin{align*}
\sum_{\ell=1}^n \frac{\sum_{i=0}^{s-2}f_{\ell,s,i}\E[\W_{\ell-1}^i]}{\prod_{j=1}^\ell f_{j,s,s}}
= \mathcal{O}\bigg(\sum_{\ell=1}^n \sum_{i=0}^{s-2}\frac{\ell^{i\Lambda}}{\ell^{\Lambda s}} \bigg)
%&= \mathcal{O}\bigg(\sum_{\ell=1}^n \sum_{i=0}^{s-2} \frac{1}{\ell^{\Lambda (s-i)}} \bigg)\\
= \mathcal{O}\bigg(\sum_{\ell=1}^n \frac{1}{\ell^{2\Lambda}} \bigg)
= \mathcal{O}(1).
\end{align*}

For triangular urns we use Remark~\ref{RemarkSimplyTriangularF} and get:
\begin{equation*}
f_{n,s,r}=\sum_{j=s-r}^s \binom{s}j \frac{\Lambda^j\sigma^j}{ m^j} p_{n;(j,j-(s-r))}
= \frac{\Lambda^{s-r}\sigma^{s-r}}{ m^{s-r}} p_{n;(j,0)}+\mathcal{O}\Big(\frac1n\Big)\\
=\mathcal{O}\Big(\frac1n\Big),
\end{equation*}
since $p_{n;(j,0)}=0$ for $j>0$ by definition of the Stirling numbers of the first and second kind. 

Concerning the second part $f_{\ell,s,s-1}\E[\W_{\ell-1}^{s-1}]$ we have to show that 
\[
\sum_{\ell=1}^n \frac{f_{\ell,s,s-1}\E[\W_{\ell-1}^{s-1}]}{\prod_{j=1}^\ell f_{j,s,s}}
\]
exists. For triangular urns we already observed the bound $f_{n,s,s-1}=\mathcal{O}(\frac1n)$.
For large urns we refine the bound $f_{n,s,s-1}=\mathcal{O}(1)$ in the following way:
\begin{equation*}
\begin{split}
f_{n,s,s-1}&=\sum_{j=1}^{s}\binom{s}j \Lambda^j\sum_{\ell=j-1}^{j}\binom{\ell}{j-1}
\Big(\frac{a_m(n-1)}{1-\Lambda}\Big)^{\ell-j+1}\sum_{i=\ell}^{j}\binom{j}{i}\frac{\sigma^i a_m^{j-i}}{m^i(\Lambda-1)^{j-i}}p_{n;(i,\ell)}\\
&= \binom{s}1 \Lambda^1\sum_{\ell=0}^{1}\binom{\ell}{0}
\Big(\frac{a_m(n-1)}{1-\Lambda}\Big)^{\ell}\sum_{i=\ell}^{1}\binom{1}{i}\frac{\sigma^i a_m^{1-i}}{m^i(\Lambda-1)^{1-i}}p_{n;(i,\ell)}+\mathcal{O}\Big(\frac1n\Big).
\end{split}
\end{equation*}
Since $p_{n;(j,0)}=0$ for $j>0$ and $p_{n,(0,0)}=1$ we get
\begin{equation*}
\begin{split}
f_{n,s,s-1}&=\binom{s}1 \Lambda^1 \bigg(\frac{a_m}{\Lambda-1}+ \frac{a_m(n-1)}{1-\Lambda}\times  \frac{\sigma}{m} \times\frac{m}{T_{n-1}}\bigg)
+\mathcal{O}\Big(\frac1n\Big).
\end{split}
\end{equation*}
We have
\begin{equation*}
\begin{split}
&\frac{a_m}{\Lambda-1}+ \frac{a_m(n-1)}{1-\Lambda} \times\frac{\sigma}{m} \times \frac{m}{T_{n-1}}
=\frac{a_m}{\Lambda-1}+ \frac{a_m\sigma(n-1)}{(1-\Lambda)(\sigma(n-1)+T_0)}\\
&\quad=\frac{a_m}{\Lambda-1}+ \frac{a_m}{1-\Lambda}\Big(1+\mathcal{O}\Big(\frac1n\Big)\Big)\\
&\quad=\mathcal{O}\Big(\frac1n\Big),
\end{split}
\end{equation*}
proving that $f_{n,s,s-1}=\mathcal{O}(\frac1n)$.
This implies that 
\[
\sum_{\ell=1}^n \frac{f_{\ell,s,s-1}\E[\W_{\ell-1}^{s-1}]}{\prod_{j=1}^\ell f_{j,s,s}}
=\sum_{\ell=1}^n \mathcal{O}\Big(\frac{1}{\ell^{\Lambda+1}}\Big)=\mathcal{O}(1).
\]
Consequently, the sums 
\[
\sum_{\ell=1}^n \frac{\sum_{i=0}^{s-1}f_{\ell,s,i}\E[\W_{\ell-1}^i]}{\prod_{j=1}^\ell f_{j,s,s}}
\]
are convergent, for $n\to\infty$ and the moments $\E[\W_n^s]\sim \kappa_s n^{s\Lambda}$. 
\end{proof}

\subsection{Higher moments for large-index urns and triangular urns}
Next we relate the limits $E_s=\lim_{n\to\infty}\E\Big[\frac{\W_n^s}{n^{s\Lambda}}\Big]$ 
of the normalized moments of $\W_n=W_n-\frac{a_m }{1-\Lambda}\, n$ with the moments 
of~$\mathcal{W}_\infty$ for large-index urns and with $\mathfrak{W}_\infty$ for triangular urn models.
\begin{prop}
\label{MULIPropMom}
For large-index urns the positive integers moments of $\mathcal{W}_\infty$ exist and can by expressed in terms of  $E_s=\lim_{n\to\infty}\E\Big[\frac{\W_n^s}{n^{s\Lambda}}\Big]$, and
%. 
\[
\E[\mathcal{W}_\infty^s]=
\sum_{k=0}^{s}\binom{s}k \frac{\Gamma^k(\frac{T_0}\sigma+\Lambda)}{\Gamma^k(\frac{T_0}\sigma)}E_k  \Big( \frac{a_mT_0}{\sigma(1-\Lambda)}-W_0\Big)^{s-k}.
\]
For triangular urn models the positive integers moments of $\mathfrak{W}_\infty$ exist and can be
expressed in terms of $E_s=\lim_{n\to\infty}\E\Big[\frac{\W_n^s}{n^{s\Lambda}}\Big]$:
\[
\E[\mathfrak{W}_\infty^s]=\frac{\Gamma^s(\frac{T_0}\sigma+\Lambda)}{\Gamma^s(\frac{T_0}\sigma)}E_s.
\]
\end{prop}
\begin{remark}
Starting with the second moments, a simple formula for $\E[\mathcal{W}_\infty^s]$ or $\E[\mathfrak{W}_\infty^s]$ seems to be elusive, 
due to the infinite sums in the expressions for $E_s$ in Theorem~\ref{TheShiftMoms}. We will present different explicit expressions for the moments in the next section.
\end{remark}

\begin{proof}[Proof]
We have
\begin{equation}
\begin{split}
\mathcal{W}_n&=g_n(W_n-\E[W_n])\\
&=g_n n^{\Lambda}\frac{W_n-n\frac{a_m}{1-\Lambda}}{n^{\Lambda}}+g_n\big( n\frac{a_m}{1-\Lambda}-\E[W_n]\big)\\
&=g_n n^{\Lambda} \frac{\W_n}{n^{\Lambda}} -g_n\Big(\E[W_n]- n\frac{a_m}{1-\Lambda}\Big).
\end{split}
\end{equation}
This implies that 
\[
\E[\mathcal{W}_n^s]=
\sum_{k=0}^{s}\binom{s}{k} (g_n n^{\Lambda})^k \E\left[\frac{\W_n^k}{n^{k\Lambda}}\right] \Big(g_n\big( n\frac{a_m}{1-\Lambda}-\E[W_n]\big)\Big)^{s-k}.
\]
We obtain from the asymptotic expansions of $g_n$ and $\E[W_n]$  in Lemma~\ref{Prop:caligraphicW}, and also the previous Lemma the stated result:
\[
\lim_{n\to\infty}g_n n^{\Lambda}=\frac{\Gamma(\frac{T_0}\sigma+\Lambda)}{\Gamma(\frac{T_0}\sigma)},
\qquad \lim_{n\to\infty} g_n\big( n\frac{a_m}{1-\Lambda}-\E[W_n]\big)=\frac{a_mT_0}{\sigma(1-\Lambda)}-W_0.
\]
\end{proof}

\section{An explicit expression for the moments}
In order to derive explicit expressions for the moments of triangular urns and large-index urns we study first the general solution of certain types of recurrence relations.
We use a correspondence between the recurrence relations and paths in weighted directed acyclic graphs in order to obtain explicit solutions using a lattice paths counting argument. 
\subsection{Triangular urns}
In the following we determine for triangular urn models an explicit expression for $\E[\tilde{W}_n^s]$ and also $E_s=\lim_{n\to\infty}\E\Big[\frac{\W_n^s}{n^{s\Lambda}}\Big]$. 
Note that $a_m=0$, so $\tilde{W}_n=W_n$.
\begin{lemma}
\label{MuliDrawsPolMomLemma1}
Let the sequence $(e_{n,s})_{n\ge 0,s\ge 1}$ be defined by the recurrence relation 
\[
e_{n,s}=\sum_{i=1}^{s}f_{n,s,i} e_{n-1,i},
\]
$n\ge 1$, with initial values $e_{0,s}=x^s$, $s\in\N_0$, for a given triple sequence $(f_{n,s,i})_{n\in\N_0,s\in\N,1\le i \le s}$. Then $e_{n,s}=\sum_{k=1}^{s}\varphi_{n,s,k}x^k$ is a polynomial in $x$ of degree $s$ with no constant term, with $\varphi_{n,s,k}$ defined by the recurrence relation
$\varphi_{n,s,k}=\sum_{\ell=k}^sf_{n,s,\ell} \varphi_{n-1,\ell,k}$, $n\ge 1$, $s\ge 1$, and $\varphi_{0,s,k}=\delta_{s,k}$.
\end{lemma}
\begin{proof}
The statement is by definition true for $n=0$ and arbitrary $s\in\N$.
Assuming the statement for values less than $n$, we obtain 
\[
e_{n,s}=\sum_{k=1}^{s}\varphi_{n,s,k}x^k=\sum_{i=1}^{s}f_{n,s,i} e_{n-1,i}=
\sum_{i=1}^{s}f_{n,s,i}\sum_{k=1}^{i}\varphi_{n-1,i,k}x^k
=\sum_{k=0}^{s}x^k\sum_{i=k}^s f_{n,s,i} \varphi_{n-1,i,k}.
\]
By comparison of coefficients of the powers of $x$, we obtain the given recurrence relations, which proves the stated result.
\end{proof}
\begin{lemma}
\label{MuliDrawsPolMomLemma2}
The coefficients $\varphi_{n,s,k}$ in the expansion of $e_{n,s}=\sum_{k=1}^{s}\varphi_{n,s,k}x^k$
are given by 
\[
\varphi_{n,s,s}=\prod_{j=1}^{n}f_{j,s,s},\quad \varphi_{n,s,s-1}=\sum_{i=1}^n\bigg(\prod_{j_1=i+1}^{n}f_{j,s,s}\bigg)f_{i_1,s,s-1}\bigg(\prod_{j_2=1}^{i_1-1}f_{j_2,s-1,s-1}\bigg),
\]
and in general, for $1\le k\le s-1$, by the expression
\[
\varphi_{n,s,k}= \sum_{\ell=1}^{s-k}\sum_{\substack{\sum_{\nu=1}^{\ell}h_\nu=s-k\\h_\nu\ge 1}}
\bigg[\sum_{1\le i_\ell<\dots<i_1\le n}\prod_{g=1}^{\ell+1}\bigg(f_{i_g,s-H_{g-1},s-H_g}\prod_{j=i_g+1}^{i_{g-1}-1}f_{j,s-H_{g-1},s-H_{g-1}}\bigg)\bigg];
\]
here $H_k=\sum_{\nu=1}^kh_\nu$, and $i_0=n+1$, $i_{\ell+1}=0$. We use the convention $f_{0,s,i}=1$.
\end{lemma}

\begin{proof}
In order to derive the stated expressions we use a lattice path counting argument. 
Given a triple sequence $(f_{n,s,i})_{n\in\N_0,s\in\N,1\le i \le s}$ we consider a weighted directed acyclic graph $G=(V,E)$ with vertices $v\in V$ identified by their pair of labels $v=(k,\ell)$, for $0\le k \le n$ and $1\le \ell \le s$. The edges $e\in E$ are directed from vertices $(k,\ell)$ to $(k-1,j)$, with $1\le k \le n$, $1\le \ell \le s$ and $1\le j\le \ell$, and an edge $e=\big( (k,\ell)\to(k-1,j)\big)$ has weight
$$
w(e)=w\big( (k,\ell)\to(k-1,j)\big))=f_{k,\ell,j}.
$$
The weight of a directed path $\mathfrak{p}$ is defined as the product of the edge weights:
$$
w(\mathfrak{p})=\prod_{e\in \mathfrak{p}}w(e).
$$ 
We have a one-to-one correspondence between the recurrence relation $e_{n,s}=\sum_{\ell=1}^{s}f_{n,s,\ell} e_{n-1,\ell}$ and certain paths in the graph $G$:
We start at the vertex $(n,s)$---the source---and end at vertices $(0,k)$, $1\le k\le s$---the sinks. 
Ending at one of the sinks $(0,k)$ corresponds to reaching the initial value $x^k$.
Te coefficients $\varphi_{n,k,s}$ appearing in the expansion $e_{n,s}=\sum_{k=1}^{s}\varphi_{n,s,k}x^k$ 
are given by the sum of weights of certain paths,
\[
\varphi_{n,k,s}= \sum_{\text{Path }\mathfrak{p}:\, (n,s)\to (0,k)}w(\mathfrak{p}).
\]
In order to obtain the weights we consider refinement of the paths from $(n,s)$ to $(0,k)$ taking into account the number of changes of the second coordinate which we call jumps. We can have~$\ell$ jumps, with $1\le \ell \le s-k$, and the individual jump heights $h_\nu$.
The total height $H_\ell=\sum_{\nu=1}^\ell h_\nu$ of the jumps has to equal $s-k$ under the restriction that $h_\nu\ge 1$, $1\le \nu \le \ell$.
Moreover, as a further refinement we fix the first coordinates $\mathbf{i}=(i_1,i_2,\dots,i_\ell)$ of the steps where the jumps occur, with $1\le i_\ell<\dots<i_1\le n$. 
Given a directed path $\mathfrak{p}$ starting at $(n,s)$ and ending at $(0,k)$ with $\ell$ jumps at steps $\mathbf{i}=(i_1,i_2,\dots,i_\ell)$ of heights $h_\nu\ge 1$, with $\sum_{\nu=1}^{\ell}h_\nu=s-k$,
the weight of such a path $\mathfrak{p}$ is given by 
\[
w(\mathfrak{p})=\prod_{g=1}^{\ell+1}\bigg(f_{i_g,s-H_{g-1},s-H_g}\prod_{j=i_g+1}^{i_{g-1}-1}f_{j,s-H_{g-1},s-H_{g-1}}\bigg).
\]
Summing over all paths $\mathfrak{p}$---taking into account all possible first coordinates $\mathbf{i}=(i_1,i_2,\dots,i_\ell)$ of the steps where the jumps occur and also 
the the different heights of the jumps---leads to the stated result.
\end{proof}

Next we combine both results to obtain an explicit representation of the moments. 
\begin{theorem}
\label{MOMProp2}
For triangular urn models the positive integer moments $\E[W_n^s]$ are given by 
\[
\E[W_n^s]=\sum_{k=1}^{s}\varphi_{n,s,k}W_0^k,
\]
with $\varphi_{n,s,k}$ as given in Lemma~\ref{MuliDrawsPolMomLemma2}. The limits $E_s=\lim_{n\to\infty} \E\big[\frac{W_n}{n^{\Lambda}}\big]^s$ of the positive integer moments of the normalized random variables $\tilde{W_n}/n^{\Lambda}=W_n/n^{\Lambda}$ can be expressed as 
\[
E_s=
\begin{cases}
\prod_{\ell=1}^{s^*}\frac{\Gamma(\frac{T_0+1-\ell}{\sigma})}{\Gamma(\lambda_{\ell,s})}
 \sum_{k=1}^{s}\tilde{\varphi}_{s,k}W_0^k, & \mom,\\
\frac{\Gamma(\frac{T_0}{\sigma})^s}{\prod_{\ell=1}^{s}\Gamma(\mu_{\ell,s})}
 \sum_{k=1}^{s}\tilde{\varphi}_{s,k}W_0^k,& \mor,
\end{cases}
\]
with $\tilde{\varphi}_{s,k}=\lim_{n\to\infty}\frac{\varphi_{n,s,k}}{\varphi_{n,s,s}}$. 
Moreover, $\tilde{\varphi}_{s,s}=1$ and for  $1\le k\le s-1$ the values $\tilde{\varphi}_{s,k}$ are convergent infinite sums:
\[
\tilde{\varphi}_{s,k}=\sum_{\ell=1}^{s-k}\sum_{\substack{\sum_{\nu=1}^{\ell}h_\nu=s-k\\h_\nu\ge 1}}
\bigg[\sum_{1\le i_\ell<\dots<i_1< \infty}
f_{i_1,s,s-H_1} \frac{\prod_{g=2}^{\ell+1}\bigg(f_{i_g,s-H_{g-1},s-H_g}\prod_{j=i_g+1}^{i_{g-1}-1}f_{j,s-H_{g-1},s-H_{g-1}}\bigg)}{\prod_{j=1}^{i_1}f_{j,s,s}}\bigg],
\]
with $f_{n,s,r}$ as given in Lemma~\ref{HIGHMOMlemma1}.
\end{theorem}
\begin{proof}
By Lemma~\ref{HIGHMOMlemma1} and \eqref{ShiftDefTriangular} the moments $\E[W_n^s]=\E[\tilde{W}_n^s]$ satisfy a recurrence relation of type discussed in Lemmas~\ref{MuliDrawsPolMomLemma1} and~\ref{MuliDrawsPolMomLemma2}, 
such that $e_{n,s}=\E[W_n^s]$, $f_{n,s,i}$ as given in Lemma~\ref{HIGHMOMlemma1}, see also Remark~\ref{RemarkSimplyTriangularF}, with initial value $x=W_0$. 
This leads to the first part of the stated result. We already know from Theorem~\ref{TheShiftMoms} that the $\lim_{n\to\infty} \E[\big(\frac{W_n}{n^{\Lambda}}\big)^s]$ exists, for $s\ge 1$. 
We also discern from the proof of Theorem~\ref{TheShiftMoms} that 
the limit 
$$
\lim_{n\to\infty}n^{-\Lambda s}\varphi_{n,s,s}=\lim_{n\to\infty}n^{-\Lambda s}\prod_{j=1}^{n}f_{j,s,s}
$$ 
exists and is a quotient of products of Gamma functions, as stated above. We factor out $\varphi_{n,s,s}$, such that
$n^{-\Lambda s}\varphi_{n,s,k}=n^{-\Lambda s}\varphi_{n,s,s}\frac{\varphi_{n,s,k}}{\varphi_{n,s,s}}$, 
and separate the factor corresponding to $g=1$ from the product 
$$\prod_{g=1}^{\ell+1}\Bigl(f_{i_g,s-H_{g-1},s-H_g}\prod_{j=i_g+1}^{i_{g-1}-1}f_{j,s-H_{g-1},s-H_{g-1}}\Bigr).$$
It follows that the nested infinite sums exist. They all share as a common factor
the quotient of products of Gamma functions. This proves the stated form of the common factor and also that
$\tilde{\varphi}_{s,s}=1$, leading to the stated result.
\end{proof}
An alternative more compact representation can be obtained by considering the discrete simplexes $$\Delta_k^n=\{\mathbf{c}=(c_1,\dots,c_n): c_i\ge 0, 1\le i\le n, \sum_{i=1}^{n}c_i=k\}.$$
Given $\mathbf{c}\in \Delta_{s-k}^n$ we interpret $c_\ell$ as the changes of the second label passing from a node with first label $\ell$ to $\ell-1$. 
Let the edges $e_\ell=e_\ell(\mathbf{c})$ be defined by 
$$e_\ell=\big((\ell,s-C_{\ell+1})\to(\ell,s-C_{\ell})\big),$$
for $1\le \ell \le n$ with $C_j=\sum_{i=j}^{n}c_i$. Then, a path $\mathfrak{p}=\mathfrak{p}(\mathbf{c})$ from $(n,s)$ to $(0,k)$ can be obtained by
$$\mathfrak{p}=(e_n,e_{n-1},\dots,e_1).$$
We extend the definition of the weight function $w$ to elements $\mathbf{c}\in\Delta_k^n$ by 
$$
w(\mathbf{c})=w(\mathfrak{p}(\mathbf{c}))=\prod_{j=1}^{n}f_{j,s-C_{j+1},s-C_j}.
$$
Consequently, the coefficients $\varphi_{n,s,k}$ can be alternatively written as
$$
\varphi_{n,s,k}=\sum_{\text{Path }\mathfrak{p}:\, (n,s)\to (0,k)}w(\mathfrak{p})=\sum_{\mathbf{c}\in\Delta_{s-k}^n}w(\mathbf{c}).
$$
This implies that
\[
e_{n,s}=\sum_{k=1}^{s}x^k\sum_{\mathbf{c}\in\Delta_{s-k}^n}w(\mathbf{c})
=\sum_{k=1}^{s}x^k\sum_{\mathbf{c}\in\Delta_{s-k}^n}\prod_{j=1}^{n}f_{j,s-C_{j+1},s-C_j}.
\]
Concerning the limits $\tilde{\varphi}_{s,k}=\lim_{n\to\infty}\frac{\varphi_{n,s,k}}{\varphi_{n,s,s}}$ appearing in $E_s$
we obtain an we readily obtain the alternative expressions:
$$\tilde{\varphi}_{s,k}=\sum_{\mathbf{c}\in\Delta_{s-k}^\infty}\prod_{j=1}^{\infty}\frac{f_{j,s-C_{j+1},s-C_j}}{f_{j,s,s}}.$$
\subsection{Large-index urns}
In order to extend the explicit results for $\E[\tilde{W}_n^s]$ to large-index urns we require a direct extension of Lemma~\ref{MuliDrawsPolMomLemma1}.
\begin{lemma}
\label{LargeLemma1}
Let the sequence $(e_{n,s})_{n\ge 0,s\ge 0}$ be defined by the recurrence relation 
\[
e_{n,s}=\sum_{i=0}^{s}f_{n,s,i} e_{n-1,i},\quad n\ge 1,s\ge 1,
\]
with initial values $e_{0,s}=x^s$, $s\in\N_0$ and $e_{n,0}=1$, $n\in\N_0$, for a given triple sequence $(f_{n,s,i})_{n\in\N_0,s\in\N_0,0\le i \le s}$. 
Then $e_{n,s}=\sum_{k=0}^{s}\varphi_{n,s,k}x^k$ is a polynomial in the variable $x$ of degree $s$ with $\varphi_{n,s,k}$ defined for $n\ge1$, $s\ge 1$ and $0\le k\le s$ by the recurrence relation $\varphi_{n,s,k}=\sum_{\ell=k}^s f_{n,s,\ell} \varphi_{n-1,\ell,k}$, and $\varphi_{n,0,k}=\delta_{k,0}$. 
\end{lemma}
%Note that $\varphi_{0,s,k}=\delta_{s,k}$ and  $\varphi_{n,0,0}=1$.

The proof of the result above is identical to Lemma~\ref{MuliDrawsPolMomLemma1} and is therefore skipped. 
Note that by the fact $e_{n,0}=1$ and $\varphi_{n,0,0}=1$ the constant term $\varphi_{n,s,0}$ satisfies
\[
\varphi_{n,s,0}=\sum_{\ell=0}^s f_{n,s,\ell} \varphi_{n-1,\ell,0}
=\sum_{\ell=1}^s f_{n,s,\ell} \varphi_{n-1,\ell,0} + f_{n,s,0}.
\]

The explicit expressions for $\varphi_{n,s,k}$, $1\le k\le s$ and $s\ge 1$ are identical to Lemma~\ref{MuliDrawsPolMomLemma2}.
What remains is to obtain an expression for the constant term $\varphi_{n,s,0}$. 

\begin{lemma}
\label{LargeLemma2}
Let the sequence $(e_{n,s})_{n\ge 0,s\ge 1}$ be defined as in Lemma~\ref{LargeLemma1}.
In the expansion $e_{n,s}=\sum_{k=0}^{s}\varphi_{n,s,k}x^k$ the coefficients $\varphi_{n,s,k}$ are for $1\le k\le s$ as explicitly stated in Lemma~\ref{MuliDrawsPolMomLemma2}.
For $k=0$ the constant term $\varphi_{n,s,0}$ is given by 
\[
\varphi_{n,s,0}=\sum_{k=1}^{n}\sum_{\ell=1}^{s} %
\sum_{r=1}^{s-\ell}\sum_{\substack{\sum_{\nu=1}^{r}h_\nu=s-\ell\\h_\nu\ge 1}}
\bigg[\sum_{k+1\le i_r<\dots<i_1\le n}\prod_{g=1}^{r+1}\bigg(f_{i_g,s-H_{g-1},s-H_g}\prod_{j=i_g+1}^{i_{g-1}-1}f_{j,s-H_{g-1},s-H_{g-1}}\bigg)\bigg];
\]
here $H_j=\sum_{\nu=1}^j h_\nu$, and $i_0=n+1$, $i_{r+1}=k$, with $H_{r+1}=s$.
\end{lemma}
\begin{proof}
In order to derive the stated expression we use again lattice path counting argument. 
Given a triple sequence $(f_{n,s,i})_{n\in\N_0,s\in\N,0\le i \le s}$ we consider a weighted directed acyclic graph $G=(V,E)$ with vertices $v\in V$ identified by their pair of labels $v=(k,\ell)$, for $0\le k \le n$ and $0\le \ell \le s$. The edges $e\in E$ are directed from vertices $(k,\ell)$ to $(k-1,j)$, with $1\le k \le n$, $1\le \ell \le s$ and $0\le j\le \ell$.
The main difference to the triangular urn models is the appearance of edges $(k,\ell)$ to $(k-1,0)$, $1\le k\le n$ and $1\le \ell \le s$, which will contribute to constant term $\varphi_{n,s,0}$.
We have a one-to-one correspondence between the recurrence relation $e_{n,s}=\sum_{\ell=0}^{s}f_{n,s,\ell} e_{n-1,\ell}$ and certain paths in the graph $G$:
We start at the vertex $(n,s)$---the source---and end at vertices $(0,\ell)$, $1\le \ell \le s$, or $(k,0)$, $0\le k\le n-1$---the sinks. 
Ending at one of the sinks $(0,k)$ corresponds to reaching the initial value $x^k$ and ending at $(k,0)$ with $0\le k\le n-1$ corresponds to reaching the initial value $1=x^0$ and contributes to the constant term.
The coefficients $\varphi_{n,k,s}$ are for $1\le k\le s$ identical to the triangular urn models
\[
\varphi_{n,k,s}= \sum_{\text{Path }\mathfrak{p}:\, (n,s)\to (0,k)}w(\mathfrak{p}).
\]
The coefficient $\varphi_{n,0,s}$ is given by the sum of weights of all path ending at one of the sinks $(k,0)$, $0\le k\le n-1$:
\[
\varphi_{n,0,s}= \sum_{k=0}^{n-1}\sum_{\text{Path }\mathfrak{p}:\, (n,s)\to (k,0)}w(\mathfrak{p}).
\]
The weight of such path are obtained by following a path $\mathfrak{p}$ from $(n,s)$ to an arbitrary  vertex $(k,\ell)$, and then end at a sink $(k-1,0)$ via the edge weighted by $f_{k,\ell,0}$, $1\le k\le n$ and $1\le \ell \le s$:
\[
\varphi_{n,s,0}=\sum_{k=1}^{n}\sum_{\ell=1}^{s} f_{k,\ell,0} \sum_{\text{Path }\mathfrak{p}:\, (n,s)\to (k,\ell)}w(\mathfrak{p}).
\]
The weight of such path can be calculated as for $\varphi_{n,s,k}$ by taking into account the position (first coordinate) of the jumps 
at $k+1\le i_r<\dots<i_1\le n$, and their heights $h_1,\dots,h_r$.
The weight of such a directed path $\mathfrak{p}$ times the weight of the last step $f_{k,\ell,0}$ is given by
\[
 f_{k,\ell,0} \cdot w(\mathfrak{p})=\prod_{g=1}^{r+1}\bigg(f_{i_g,s-H_{g-1},s-H_g}\prod_{j=i_g+1}^{i_{g-1}-1}f_{j,s-H_{g-1},s-H_{g-1}}\bigg),
\]
where we have used the convention $i_0=n+1$, $i_{r+1}=k$, with $H_{r+1}=s$. Taking into account all such path leads to the stated result. 
\end{proof}
\begin{theorem}
\label{MOMProp3}
For large-index urns the positive integer moments $\E[\tilde{W}_n^s]$ of the shifted random variable $\tilde{W}_n=W_n-\frac{n a_m}{1-\Lambda}$ are given by 
\[
\E[\tilde{W}_n^s]=\sum_{k=0}^{s}\varphi_{n,s,k}W_0^k,
\]
with $\varphi_{n,s,k}$ for $1\le k\le s$ as given in Lemma~\ref{MuliDrawsPolMomLemma2} and $\varphi_{n,s,0}$ as given in Lemma~\ref{LargeLemma2}.
The limits $E_s=\lim_{n\to\infty} \E\big[\frac{\tilde{W}_n}{n^{\Lambda}}\big]^s$ can be expressed as 
\[
E_s=
\begin{cases}
\prod_{\ell=1}^{s^*}\frac{\Gamma(\frac{T_0+1-\ell}{\sigma})}{\Gamma(\lambda_{\ell,s})}
 \sum_{k=0}^{s}\tilde{\varphi}_{s,k}W_0^k, & \mom,\\
\frac{\Gamma(\frac{T_0}{\sigma})^s}{\prod_{\ell=1}^{s}\Gamma(\mu_{\ell,s})}
 \sum_{k=0}^{s}\tilde{\varphi}_{s,k}W_0^k,& \mor,
\end{cases}
\]
with $\tilde{\varphi}_{s,k}$ for $1\le k\le s$ as stated in Theorem~\ref{MOMProp2}. 
Moreover, $\tilde{\varphi}_{s,0}$ is a convergent infinite sum:
\[
\tilde{\varphi}_{s,0}=\sum_{k=1}^{\infty}\sum_{\ell=1}^{s} %
\sum_{r=1}^{s-\ell}\sum_{\substack{\sum_{\nu=1}^{r}h_\nu=s-\ell\\h_\nu\ge 1}}
\bigg[\sum_{k+1\le i_r<\dots<i_1<\infty}\frac{\prod_{g=2}^{r+1}\bigg(f_{i_g,s-H_{g-1},s-H_g}\prod_{j=i_g+1}^{i_{g-1}-1}f_{j,s-H_{g-1},s-H_{g-1}}\bigg)}{\prod_{j=1}^{i_1}f_{j,s,s}}\bigg];
\]
\end{theorem}

\section{Conclusion and Outlook}
\subsection{Summary}
We studied for a part I and II two-color affine linear urn models with multiple drawings---sample size $m\ge 1$---under two sampling models the distribution of the number of white balls $W_n$ after $n$ draws. Concerning the distribution of the number of white balls $W_n$ we obtained several limit laws summarized in the Theorem below.
\begin{theorem}[Limit laws for affine balanced two-color urn models]
\label{TheFolklore2}
For a balanced two-color affine urn model with sample size $m\ge1$, let $\Lambda$ denote the ratio of the two eigenvalues of $M$. 
\begin{enumerate}
	\item Small-index urns, the case $\Lambda\le\frac12$: $\frac{W_n-\E(W_n)}{\sqrt{\V(W_n)}}\claw \mathcal{N}(0,1)$. 
	\item Large-index urns, the case $\Lambda>\frac12$: $\frac{W_n-\E(W_n)}{n^{\Lambda}}\as L$.
	\item Triangular urns, the case $b_0\cdot a_1=0$: $\frac{W_n}{n^{\Lambda}}\as T$. 
\end{enumerate}
For large-index and triangular urn models we have convergence of all positive integer moments. 
\end{theorem} 
The central limit theorem for small-index urn is obtained in part I of Theorem 3. 
The random variables $L$ and $T$ are variants of the almost-sure limits $\mathcal{W}_\infty$ and $\mathfrak{W}_\infty$ and their almost-sure convergence can be easily deduced from part I of Theorem 2 and Proposition 4. The convergence of the moments follows directly from Proposition~\ref{MULIPropMom} and Theorems~\ref{MOMProp2} and~\ref{MOMProp3}.

\subsection{Open problems and extensions}
A natural question is to extend the results to nonaffine linear urn models with multiple drawings, both balanced and unbalanced, and to extend the general limit theorems for $m=1$ of Janson~\cite{Jan2004,Jan2006}. Moreover, it is also of interest to extend the analytic combinatorial framework developed for sample size $m=1$ by Flajolet et al.~\cite{FlaDumPuy2006,FlaGabPek2005} and Morcrette~\cite{Morcrette2013} to urn models with multiple drawings, both balanced and unbalanced to obtain precise information about $W_n$ for fixed $n$. Another important question is the study 
of the limit law of triangular urn models and large-index urns for $m>1$ similar to the case $m=1$, which is now well understood~\cite{Chauvin1,Chauvin2,NeiningerKnape} in terms of fixed-point equations (smoothing transforms). Does there exist a simpler expression for the moments?

\smallskip

The methods applied in this work do not seem to be easily adapted to unbalanced urn models. See for example~\cite{Aguech2014} for a first step towards the analysis of such models. However, extensions to more than two colors and to certain classes of nonaffine models can be obtained. For example, one can readily generalize the affine linearity condition~\eqref{MuliDrawsLinPropLinear} and the martingale structure of Proposition~\ref{Prop:caligraphicW} to balanced urn models with $r\ge 2$ colors. Moreover, using the correspondence between recurrence relations for integer moments and weighted directed acyclic graphs it seems possible to study more general classes of two-color balanced urn models. It is possible to extend the approach of Flajolet et al.~\cite{FlaDumPuy2006} for balanced urns and of Morcrette~\cite{Morcrette2013} for unbalanced urns to derive partial differential equations for suitably defined generating functions; we comment on this elsewhere~\cite{KuMor2014+}.

\end{document}